\documentclass[11pt]{article}
\usepackage{latexsym,amsmath,amsthm,verbatim,ifthen,amssymb}
\usepackage[latin1]{inputenc}
\usepackage[all]{xy}
\usepackage{graphicx}
\usepackage{epsfig}

\newdir{ >}{{}*!/-12pt/@{>}}
\newdir{ (}{!/-3pt/@_{(} } 

\newdir{ )}{!/-3pt/@^{(} } 

\newcommand{\ex}{\mathcal{E}} 
\newcommand{\T}{\tau}  

\newcommand{\timess}{\bar \times}

\newcommand{\E}{\textbf{E}}

\newcommand{\R}{\mathbb{R}}

\newcommand{\exter}{\mathcal{E}}

\newcommand{\N}{\mathbb{N}}

\newcommand{\Nmed}{\mbox{{\scriptsize $\N$}}}

\newcommand{\SN}{\mathfrak{S}}
\newcommand{\SB}{\mathfrak{S}_{\mbox{{\tiny $\!\mathcal{B}$}}}}

\newcommand{\DN}{\mathfrak{D}}

\newcommand{\piB}{\pi^{\mbox{{\tiny $\mathcal{B}$}}}}

\newtheorem{theorem}{Theorem}[section]

\newtheorem{corollary}[theorem]{Corollary}
 \newtheorem{lemma}[theorem]{Lemma}
 \newtheorem{proposition}[theorem]{Proposition}
 \theoremstyle{definition}
 \newtheorem{definition}[theorem]{Definition}
 \theoremstyle{remark}
 \newtheorem{remark}[theorem]{Remark}
 
 \numberwithin{equation}{subsection}

\newcommand{\PP}{\textbf{P}}

\begin{document}

\title{Brown representability for exterior cohomology and cohomology
with compact supports\footnotetext{ This work has been supported by
the Ministerio de Educaci\'on y Ciencia grants MTM2009-12081 and
MTM2010-18089, and by the Junta de Andaluc\'{\i}a grant FQM-213.\vskip 4pt
2010 \textit{Mathematics Subject Classification}: 55P57, 55N99.
\vskip 4pt
\textit{Keywords}: Proper homotopy, cohomology with
compact supports, exterior space, exterior cohomology, Brown
representability, Brown-Grossman homotopy groups.}}

\author{J. M. Garc\'{\i}a-Calcines, P. R. Garc\'{\i}a-D\'{\i}az \\
and A. Murillo}

\maketitle

\begin{abstract}
It is well known that cohomology with compact supports is not a
homotopy invariant but only a proper homotopy one. However, as the
proper category lacks of general categorical properties, a Brown
representability theorem type does not seem reachable. However, by
proving such a theorem for the so called exterior cohomology in
the complete and cocomplete exterior category, we show that the
$n$-th cohomology with compact supports of a given countable,
locally finite, finite dimensional relative CW-complex $(X,\R_+)$
is naturally identified with the set $[X,K_n]^{\R_+}$ of exterior
based homotopy classes from a ``classifying space'' $K_n$. We also
show that this space has the exterior homotopy type of the
exterior Eilenberg-MacLane space for Brown-Grossman homotopy
groups of type $(R^\infty,n)$, $R$ being the fixed coefficient
ring.
\end{abstract}

\section*{Introduction}
Proper homotopy theory is designed to study  non compact
topological spaces modulo deformations which respect the behaviour
of these spaces at infinity. In this paper we are concerned with
the classification of cohomology invariants of proper homotopy.
Among them, classical cohomology with compact supports and locally
finite cohomology \cite[\S3]{huguesranicki} are specially well
adapted functors to study  locally compact spaces up to proper
homotopy.

In classical homotopy theory, the representation, in the Brown
sense, of cohomology  functors in terms of homotopy classes of
maps coming from a universal space was the starting point of many
important developments. Proper homotopy theory would also very
much benefit from this classification. However, due to the lack of
categorical properties of the proper category $\mathbf{P}$, the
behaviour of classical homotopy invariants are not easy, when
possible, to translate to the proper  setting. This applies to the
mentioned problem: first of all, the above cohomology theories are
not homotopy invariants but only proper homotopy invariants, so
classical Brown representability cannot be attained within this
category. Indeed, none of the various representability theorems
available applies as $\mathbf{P}$ is not closed even for finite
limits or colimits.   Recall that the proper category is only a
cofibration category \cite{Ay-D-Q} and that the most general forms
of Brown representability theorem, see for instance
\cite{jardine,neeman,rosicky}, requires the category to which it
applies to be, among other conditions, cocomplete.

We overcome these obstructions via the following procedure:
consider the full embedding of $\mathbf{P}$ into the so called
{\em exterior category} $\mathbf{E}$ (see next section for precise
definition and details), a complete and cocomplete category in
which exterior homotopy theory can be developed \cite{Ext_1}. We
can also consider the {\em exterior cohomology} in $\mathbf{E}$,
an extension of compactly supported cohomology to the exterior
category. This already appears in \cite[\S2]{spanier4} under a
different approach, and it has recently been used to obtain
interesting results in proper homotopy theory
\cite{sevillanos1,sevillanos2,garcalpari}. Finally, we show that
the exterior cohomology satisfies all the necessary properties to
be classified in the Brown sense.

As a result of this procedure we prove the following
representability theorem for the exterior cohomology of the, so
called, exterior CW-complexes, (see next sections for precise
definitions and notation).

\begin{theorem}\label{intro}
For each $n\geq 0$, there exist an  e-path connected, exterior
CW-complex $K_n\in \mathbf{Ho CW}^{\mathbb{R}_+}$, unique up to
based exterior homotopy, and a universal element $u\in
H_\exter^n(K_n)$, such that the natural natural transformation
$$T_u:[-,K_n]^{\mathbb{R}_+}\longrightarrow H_\exter ^n(-),\qquad T_u[f]=H_\exter^n(f)(u),$$
induces a bijection $ [X,K_n]^{\mathbb{R}_+}\cong H_\exter^*(X) $
for any exterior CW-complex $X\in \mathbf{Ho CW}^{\mathbb{R}_+}.$

\end{theorem}

The translation of this result to the classical cohomology with
compact supports on the proper category reads:

\begin{theorem} \label{intro2} Let $(X,\R_+)$ be a countable,
locally finite, finite dimensional relative CW-complex, and let
$n\ge 0$. There exist an e-path connected, exterior CW-complex
$K_n\in \mathbf{Ho CW}^{\mathbb{R}_+}$ (not necessarily endowed
with the cocompact externology!), unique up to based exterior
homotopy, and a universal element $u\in H_\exter^n(K_n)$, such
that the map
$$T_u:[X,K_n]^{\mathbb{R}_+}\stackrel{\cong}{\longrightarrow} H_c^n(X),\qquad T_u[f]=H_\exter^n(f)(u),$$
is a bijection.
\end{theorem}

In other words, cohomology with compact supports of one ended
spaces in the proper category are classified by homotopy classes of
based exterior maps into a classifying space which generally lives
outside the proper category.

We finish by describing, as in the classical case, the exterior
homotopy type of the classifying space for the exterior cohomology
as an Eilenberg-MacLane space for Brown-Grossman homotopy groups.
See Theorem \ref{descripcion}.

In the next section we present a brief summary of exterior
homotopy theory. In Section 2 we set the main properties of the
exterior cohomology by regarding it as the cohomology of the so
called {\em Alexandroff construction} of the given exterior space.
In Section 3 we prove Theorems \ref{intro}, \ref{intro2} and
explicitly describe the exterior homotopy type of the classifying
space for exterior cohomology.

\section{Exterior homotopy theory}

For a brief review of fundamental results on exterior homotopy
theory we refer to \cite[\S1]{G-G-M2}. Here we simply recall the
basic concepts and facts we will use.

A \textit{proper} map is a continuous map
$f:X\rightarrow Y$ such that $f^{-1}(K)$ is a compact subset of
$X,$ for every closed compact subset $K$ of $Y.$ We will denote by
$\mathbf{P}$ the category of spaces and proper maps.
In proper homotopy theory the role of the base point is played by
the half-line $\mathbb{R}_+=[0,\infty).$

 An \textit{exterior space}
$(X,\ex)$ consists of a topological space $(X,\T)$ together with a
non empty family of {\em exterior sets} $\ex\subset \T $, called
\textit{externology} which is closed under finite intersections
and, whenever $U \supset E$, $E \in \ex$, $U\in \T$, then $U\in
\ex$.
 We may think of $\ex$ as a  neighborhood system at infinity. A map $f:(X,\ex ) \rightarrow (X',\ex'
)$ is  \textit{exterior} if it is continuous and $f^{-1}(E) \in
\ex$, for all $E \in \ex'$. A subset $A\subset X$ is called an
{\em exterior neighborhood} if it contains an exterior set. The
complement of an exterior set is called an exterior closed set or
e-closed. The category of exterior spaces shall be denoted by
$\E$.

For a given topological space $X$ we can consider its
\textit{cocompact externology} $\ex_{cc}$ which is formed by the
family of the complements of all closed-compact sets of $X$. The
corresponding exterior space will be denoted by $X_{cc}$. The
correspondence $X\mapsto X_{cc}$ gives rise to a full embedding
\cite[Thm. 3.2]{Ext_1}
$$(-)_{cc}\colon \PP \hookrightarrow \E.$$
Furthermore, the category $\E$ is complete and cocomplete
\cite[Thm. 3.3]{Ext_1}.

Let $X$ and $Y$ be an exterior and a topological space
respectively. On the product space $X\times Y$ consider the
following externology: an open set $E$ is exterior if for each
 $y\in Y$ there exists an open neighborhood of  $y$, $U_y$, and an exterior open
 $E_y$ such that $E_y \times U_y \subset E$.
We denote by $X\timess Y$  the resulting exterior space. If $Y$
is compact, then $E$ is an exterior open if and only if it is an
open set and there exists $G\in \ex_X$ for which $G\times Y\subset
E$. In particular, if $\ex_X=\ex_{cc}^X$ and $Y$ is compact, then
$X_{cc}\timess Y=(X\times Y)_{cc}$. Hence, the {\em cylinder}
 $$- \timess I\colon\E \rightarrow
\E,$$
 together with the obvious natural transformations
 $$\imath_0, \, \imath_1\colon id
\rightarrow - \timess I,\qquad\rho\colon - \timess I \rightarrow
id,$$ provide a natural way to define {\em exterior homotopy} in
$\E$. This functor restricts to the proper category
$$- \timess I=- \times I\colon\mathbf{P} \rightarrow
\mathbf{P}.$$

It is worth mentioning that, unlike in the proper framework
\cite[Thm. 1.4]{Ay-D-Q}, the exterior cylinder has a right adjoint
\cite[Thm. 3.5]{Ext_1}
 $$
 (- )^I\colon \E \rightarrow \E
 $$
which also leads, this time via paths, to the same notion of
exterior homotopy.

Given exterior spaces $X,Z$, the mapping set $Z^X$ of exterior
maps is canonically endowed with the topology generated by
  $$(K,U)=\{\alpha \in Z^X,\,\alpha (K)\subset
  U\}\hspace{5pt}\mbox{and}\hspace{5pt}
  (L,E)=\{\alpha \in Z^X,\,\alpha (L)\subset E\}$$ \noindent
  where $K$ is compact, $U$ is open, $E$ is exterior, and L is e-compact,
that is, $L\setminus F$ is compact for any exterior $F$. Given a
Hausdorff, locally compact space $X$   endowed with the cocompact
externology,  there is a natural bijection
$$
Hom_E(X\bar{\times }Y,Z)\cong Hom_{{Top}}(Y,Z^X).
$$
The right framework for pointed exterior homotopy is the category
$\E^{\R_+}$ of {\em based exterior spaces} or {\em exterior spaces
under $\mathbb{R}_+$} in which $\mathbb{R}_+$ is endowed with the
cocompact externology. Its objects are pairs $(X,\alpha),$ where
$\alpha\colon \R_+ \rightarrow X$ is an exterior map (the
\textit{base ray}). Morphisms $f\colon (X,\alpha ) \rightarrow
(Y,\beta )$ are exterior maps $f\colon X\rightarrow Y$ for which
$f \alpha = \beta .$

From now on $\N\subset \R_+$ will always be endowed with the
induced externology. Also, any compact space is considered with
the topology as externology. Note that any $(X,\alpha)\in
\E^{\R_+}$ may, and will henceforth, be considered in the category
$\E^\N$ of exterior spaces under $\N$ by composing $\alpha$ with
the exterior inclusion $\N\hookrightarrow \R$.

 Given
$(X,\alpha )\in \E^{\R_+}$, the \textit{based exterior cylinder of
$X$}, $I^{\R_+} X$, is defined by the pushout in $\E:$
$$
\xymatrix{
  \R_+\timess I \ar[d]_{\alpha\timess id} \ar[r]^{\rho} & \R_+ \ar[d]^{} \\
  X\timess I \ar[r]^{} & I^{\R_+} X  }
$$
The notion of {\em based exterior homotopy} is then naturally
defined in the obvious way. This can be dually introduced by means
of the \textit{based exterior cocylinder of $X$}.

We also need to  recall the notion of well based exterior spaces. An exterior map $j\colon A \rightarrow X$
is an \textit{exterior cofibration} if it satisfies the Homotopy
Extension Property {\em (HEP)} in $\E$. That is, if for any
commutative diagram of exterior maps
$$\xymatrix{
  A \ar[d]_{j} \ar[r]^{\imath_0} & A\timess I \ar[d]_{}
  \ar@/^/[ddr]^{j\timess id}  \\
  X \ar[r]^{} \ar@/_/[drr]_{\imath_0}   & Y             \\
                &               & {X\timess I} \ar@{.>}[lu]}$$ the
dotted arrow always exists. In the proper setting, if $j$ is a
proper map between Hausdorff and locally compact spaces, then  $j$
is a proper cofibration if and only if $j_{cc}$ is an exterior
cofibration. Recall that a proper cofibration is a proper map
which satisfies the corresponding proper homotopy extension
property.

A based exterior space $(X,\alpha)$ is {\em well based} if $\alpha
$ is an exterior closed cofibration. We denote the corresponding
full subcategory by  $\E^{\R_+}_w\subset \E^{\R_+}$. We point out
that $\E^{\R_+}_w$ verifies all the axioms required for a closed
model category, except for being closed for finite limits and
colimits  \cite[Thm. 2.12]{G-G-M}. Fibrations (resp. cofibrations)
are  exterior based maps which verify the Homotopy Lifting
Property (resp. the Homotopy Extension Property), while  weak
equivalences are  based exterior homotopy equivalences.
Nevertheless, many of such limits, as pullbacks of fibrations and
pushouts of cofibrations, can be constructed within $\E_w^{\R_+}$
and thus, {\em exterior homotopy pullbacks and pushouts} are
defined within this category.

Next, we recall the notion of {\em exterior CW-complex} we use. It
is slightly more general than the original one introduced in
\cite{G-G-M2}, which include for instance, the classical
CW-complexes (with topology as externology) and, in the proper
case,  the spherical objects under a tree of \cite[\S 2]{B_pre},
\cite[IV]{B-Q} and \cite[1.3]{B-Q2}.

Given $n\ge 0$, we denote by  $\SN^k$  either the classical $k$-dimensional
sphere $S^k$ or the
$k$-dimensional $\N$-sphere $\N\timess S^{k}.$ Analogously $\DN^k$
will ambiguously denote either the classical $k$-dimensional disk $D^k$
or $\N\timess D^{k},$ the $k$-dimensional $\N$-disk. The inclusions
 $\SN^{k-1} \hookrightarrow \DN^{k}$ are exterior cofibrations.

A \textit{relative exterior CW-complex} $(X,A)$ is an exterior
space $X$ together with a filtration of exterior subspaces
$$A=X^{-1} \subset X^0 \subset X^1 \subset \ldots \subset
X^n\subset \ldots \subset X$$ \noindent for which
$X=\mbox{colim}\hspace{3pt}X^n,$ and for each $n\geq 0$, $X^n$ is
obtained from $X^{n-1}$ as the exterior pushout
$$\xymatrix@C=1.6cm@R=1cm{
  \sqcup_{\gamma \in \Gamma} \SN^{n-1}_\gamma \ar@{^{ (}->}[d]_{} \ar[r]^{\sqcup_{\gamma \in \Gamma}\varphi_\gamma} & X^{n-1} \ar@{^{ (}->}[d]^{} \\
  \sqcup_{\gamma \in \Gamma} \DN^{n}_\gamma \ar[r]^{\sqcup_{\gamma \in \Gamma}\psi_\gamma} & X^n   }$$
via the  attaching maps $\varphi_\gamma\colon \SN^{n-1} \rightarrow
X^{n-1}$. The resulting maps $\psi_\gamma$ are called
{\em characteristic} and the exterior spaces $\psi_\gamma(\DN^n_\gamma)$
are the {\em cells} of dimension $n$ of $X$. If $A=\emptyset $, then $X$
is simply called \emph{exterior CW-complex}. On the other hand,  if
$A=\mathbb{R}_+$, then $(X,\mathbb{R}_+)$ is  called a
\emph{based exterior CW-complex}.

Note that $(X,\mathbb{R}_+)\in \E^{\R_+}_w$ as the inclusion
$\mathbb{R}_+\hookrightarrow X$ is a closed cofibration. We will
denote by $\mathbf{CW}^{\mathbb{R}_+}$, (resp.
$\mathbf{CW}^{\mathbb{R}_+}_f$)  the full subcategory of
$\mathbf{E}^{\mathbb{R}_+}_w $ formed by based exterior
CW-complexes (resp. finite based exterior CW-complexes). We denote
by $ \mathbf{Ho CW}^{\mathbb{R}_+}$ (resp. $ \mathbf{Ho
CW}^{\mathbb{R}_+}_f$) its corresponding homotopy category with
the same objects, and whose set of morphisms $[X,Y]^{\R_+}$
between $(X,\R_+)$ and $(Y,\R_+)$, or simply $X$ and $Y$ for
simplicity, are based exterior homotopy classes of exterior maps.
Remark that a finite exterior CW-complex is in general a finite
dimensional infinite classical CW-complex. Also,  any classical
CW-complex is an exterior CW-complex with its topology as
externology. Other important class of exterior CW-complexes are
constituted by the open differential manifolds and PL-manifolds as
they admit a locally finite countable triangulation, which
describes the exterior CW-structure \cite[\S 2(ii)]{G-G-M2}.

On the other hand, if $(X,A)$ is any countable, locally finite,
finite dimensional relative CW-complex in which $A$ is a Hausdorff
locally compact space, then it is not difficult to check that
$(X_{cc},A_{cc})$ is a relative exterior CW-complex
\cite{garcalpari}, \cite[Prop. 2.3]{G-G-M}. As we will use them
later, we point out that there are exterior versions of the
Whitehead and cellular approximation theorems for exterior
CW-complexes \cite[Thms. 12,13,14]{G-G-M2}, \cite[Thms.
4.1.19,4.1.26]{Thesis}.

We finish this section by recalling how homotopy groups are
introduced in the exterior setting. For it, consider $\N$-spheres
$\SN^k=\N \bar{\times}S^k$ and $\N$-disks $\DN^{k+1}=\N
\bar{\times}D^{k+1}$ as spaces in $\E^{\N}$ via the natural
exterior inclusions  $\eta\colon\N \hookrightarrow \SN \subset \DN
$, $\eta(n)=(n,*)$.  Given any $(X,\alpha)\in \E^{\R_+}$ (or more
generally, in $\E^{\N}$),  and any $k\ge 0$, the {\em $k$-th
Brown-Grossman exterior homotopy group of $(X,\alpha)$}
\cite[2.3]{Ext_2} is defined as
$$\piB_k(X,\alpha)=[(\SN^k, \eta), (X,\alpha)]^{\Nmed} \: $$
\noindent where $[-,-]^{\Nmed}$ denotes the set of homotopy
classes under $\N .$
The group structure, for $k\ge1$, is given by the natural bijection
$$[(\SN^{k}, \eta), (X,\alpha)]^{\Nmed} \cong [(S^{k}, *),
(X^{\Nmed},\alpha)]^*$$
with the
ordinary $k$-th homotopy group of the pair
 $(X^{\Nmed},\alpha)$ (see \cite[Prop. 3.2]{Ext_1} or \cite[Rem. 2.4]{Ext_2}).

 A ``continuous'' way to see these homotopy groups is the following
 (see for instance \cite[Rem. 7]{G-G-M2}). For each $n\ge 1$ consider the based,
 finite, exterior CW-complex $S^n_\eta$ obtained as the pushout
$$\xymatrix{\N \ar@{^{(}->}[d]
\ar@{^{(}->}[r]^\eta & \SN^n \ar@{^{(}->}[d] \\
\R_+ \ar@{^{(}->}[r]_\xi & S^n_\eta.
}
$$
That is, $S^n_\eta$ is the half real line, with an $n$-sphere
attached to any integer number, and endowed with the cocompact
externology. For any based exterior space $(X,\alpha)$, the
universal property of the pushout provides a canonical isomorphism
$$[(S^n_\eta,\xi), (X,\alpha)]^{\R_+} \cong [(\SN^n,\eta),(X,\alpha)]^\N=\piB_n(X,\alpha).$$

An exterior space $X\in \E$ is {\em exterior $n$-connected} if
it is $n$-connected in the classical sense, and $\piB_k(X,\alpha
)=\{0\}$, $0\leq k\leq n$, for any exterior base sequence
$\alpha\colon \N \rightarrow X$. The space $X$ is {\em e-path connected} if it is exterior $0$-connected.

An exterior map $f\colon Y \rightarrow Z$ is an
\textit{exterior $n$-equivalence}  if it is a classical
$n$-equivalence and $f_*\colon \piB_k(Y,\alpha) \longrightarrow
\piB_k(Z,f \alpha)$ is isomorphism for $0\leq k \leq n-1$ and surjective
for $k=n$, for any exterior base sequence $\alpha:\N \rightarrow
Y$.

\section{Exterior cohomology}
In this section, as stated in the introduction, we set the main
properties of {\em exterior cohomology}, already introduced in
\cite[\S2]{spanier4} under a different approach, and successfully
used in \cite{sevillanos1,sevillanos2,garcalpari}. From now on,
and unless explicitly stated otherwise, we fix a continuous and
additive classical cohomology theory, i.e.,  a contravariant
functor $H^*$ from the category of pairs of spaces to the category
of non negatively graded abelian groups satisfying {exactness,
excision, homotopy invariance, continuity (tautness), and
additivity. By continuity  we mean the following: for every triad
$(X,A,B)$ in which $(A,B)$ is a closed pair, $H^*(A,B)=\varinjlim
H^*(Y,Z)$, with $(Y,Z)$ closed neighborhood of $(A,B)$ in $X$. On
the other hand, additivity means that, for  any family
$\{X_i\}_{i\in I}$, $H^*(\coprod_{i\in I} X_i)\cong \prod_{i\in
I}H^*(X_i)$.

\begin{remark} Our cohomology set of axioms may not be the most
general one, see for instance \cite{spanier5}, to which what
follows can be applied. However, classical theories like \v{C}ech,
or Alexander-Spanier cohomology, fit within our framework.
Moreover, it is worth remarking that both of them coincide with
singular cohomology when we restrict to {\em HLC spaces}
\cite[Chap. III]{bredon}. Recall that a space $X$ is {\em
homologically locally connected} (HLC) if for every  neighborhhod
$U$ of a given point $x\in X$ there is another neighborhood
$V\subset U$ of $x$, such that the morphism induced in reduced
singular homology $\widetilde H_*(V)\to \widetilde H_*(U)$ is
trivial. In particular, CW-complexes, or more generally, locally
contractible spaces are HLC and, for them, our main results hold
even when considering singular cohomology theory.
\end{remark}

\begin{definition}\label{exteriorcohomology}  Let $(X,\mathcal{E})$ be
an exterior space. Define the {\em exterior cohomology} of $X$ as,
$$H_\mathcal{E}^*(X)=\varinjlim \{H^*(X,E),\,\, E\in\mathcal{E}\}.$$
This defines a contravariant functor
$$H^*_\mathcal{E}\colon\E\longrightarrow\mathbf{Sets}.$$
\end{definition}

\begin{remark} (i) Observe  that this definition coincides in most cases
with the one in \cite[Thm. 3.2]{spanier4}. Indeed, a {\em family
of supports} $\phi$ of a space $X$, a concept which goes back to
\cite[I.6]{bredon}, is precisely formed by the complements of the
elements of an externology $\mathcal{E}$ of $X$, i.e.,
$\Phi=\{E^c,\,\,E\in\mathcal{E}\}$.  Moreover, it is
straightforward to show  that a {\em co-$\Phi$ set} of
\cite{spanier4} is nothing but an exterior neighborhood in our
terminology.

(ii) For the special case of singular cohomology, exterior
cohomology was already introduced in \cite{sevillanos1} by
considering, for each $E\in\mathcal{E}$,  the  complex $C^*(X,E)$
of singular cochains of $X$ which vanish in $E$.

(iii) Note also that, choosing $\mathcal{E}=\mathcal{E}_{cc}$ the
cocompact externology, $H_\mathcal{E}^*(X)=H_c^*(X)$ is the
classical compact supported cohomology of $X$. Moreover, see for
instance \cite[Prop. 3.6]{spanier3}, whenever $X$ is Hausdorff and
locally compact,   $H_c^*(X)$ coincide with the reduced cohomology
of the Alexandroff one-point compactification $X^+$ of $X$, that
is,
\begin{equation}\label{ecuacion1}
H_c^*(X)=\ker\bigl(H^*(X^+)\to H^*(\infty)\bigr)=\widetilde
H^*(X^+)
\end{equation}
\end{remark}

In order to show important properties of exterior cohomology, an
analogous equation comparing the exterior cohomology of a given
exterior space $X$ with the reduced cohomology of the so called
{\em Alexandroff exterior construction} of $X$ will be considered.
We explicitly recall this construction, studied in \cite{GC} to
extend the classical result of Dold on partitions of unity
\cite[Thm. 6.1]{dold}, to the proper category.

\begin{definition}\label{alexandroff} Given an exterior space
$(X,\mathcal{E})$ with topology $\tau$, the {\em Alexandroff
exterior construction of $X$} is the based topological space
$X^\infty$ defined as the disjoint union of $X$ with the based
point $\infty$, and endowed with the topology $\tau^\infty=\tau
\cup \{E\cup\{\infty\},\,\,E\in \mathcal{E}\}$.
\end{definition}

If $\mathbf{Top^*}$ denotes the category of pointed spaces and
maps, the above construction defines a functor
$$(-)^{\infty }\colon{\mathbf{E}}\longrightarrow {\mathbf{Top}}^*,$$
for which the following holds.

 \begin{proposition}\label{propiedades}
The functor $(-)^{\infty }$ preserves:
\begin{enumerate}
\item[(i)] Small limits and colimits. In particular
$(X\cup Y)^\infty=X^\infty\cup Y^\infty$ and
$(X\cap Y)^\infty=X^\infty\cap Y^\infty$ for any $X,Y\in\mathbf{E}$.

\item[(ii)] Cylinders, that is,  $(X\bar{\times }I)^{\infty }\cong (X^{\infty }\times
I)/\{\infty\}\times I$, for any $X\in\mathbf{E}$.
In particular, $(-)^{\infty }$ preserves homotopies.

\item[(iii)] Cofibrations and homotopy equivalences, i.e.,
for any exterior cofibration (resp. exterior homotopy
equivalence)  $f\colon X\rightarrow Y$, the map $f^{\infty }\colon
X^{\infty }\rightarrow Y^{\infty }$ is a classical pointed
cofibration (resp.  pointed homotopy equivalence).

\item[(iv)] Homotopy colimits.

\end{enumerate}
\end{proposition}
\begin{proof} For (i),(ii)  and  preservation of homotopy
equivalences we refer to  \cite[\S2,3]{GC}. Preservation of
cofibrations and homotopy colimits are then immediate
consequences.
\end{proof}

\begin{corollary}\label{infinitocw}
In particular, $(-)^{\infty }$ induces a functor denoted in the same way
$$
(-)^{\infty }\colon\mathbf{Ho CW}^{\mathbb{R}_+}\longrightarrow {\mathbf{HoTop}}^*.
$$
\hfill$\square$
\end{corollary}
 We also point out  that, for any   $X\in \mathbf{Ho CW}^{\mathbb{R}_+}$,  $X^{\infty
}$ is Hausdorff and paracompact \cite[Prop.14]{GC}.

Next, we show that Equation \ref{ecuacion1} also holds in the exterior setting.

\begin{theorem}\label{igualdadcohomologia}
For any exterior space $(X,\mathcal{E})$,
$$H^*_\mathcal{E}(X)\cong \widetilde H^*(X^\infty)=\ker\bigl(H^*(X^\infty)\to H^*(\infty)\bigr).
$$
\end{theorem}

\begin{proof}
First, observe that $X^\infty$ is an exterior space endowed with
the externology $\mathcal{E}^\infty$ given by the open
neighborhoods of $\infty$, that is
$\mathcal{E}^\infty=\{E\cup\{\infty\},\,\,E\in\mathcal{E}\}$.
Moreover, the inclusion $X\hookrightarrow X^\infty$ is an exterior
map.  Hence, by the continuity of $H^*$,  considering the triad
$(X^\infty,X,\emptyset)$, and taking into account that $X^\infty$
is the only closed exterior neighborhood of $X$,
$$
H^*_\exter(X)=\varinjlim H^*(X,E)\cong\varinjlim H^*(X^\infty,B),
$$
where $B$ ranges over the family of closed exterior neighborhoods
of $X^\infty$ (that is, $B$ is closed in $X^\infty$ and contains
an exterior set, also of $X^\infty$). But this family is precisely
that of all closed neighborhoods of the point $\infty$ which is
itself closed. Again by the continuity of $H^*$,
$$
\varinjlim H^*(X^\infty,B)\cong H^*(X^\infty,\infty)
$$
and the theorem is proved.
\end{proof}

\section{Brown representability of exterior cohomology}

In the most general categorical framework, Brown representability
for set-valued  contravariant functors, defined on the homotopy
category of a given closed model category, is now classical and
well understood. See \cite{chriskenee,neeman,rosicky}, or
\cite[Thm. 19]{jardine} for a particularly  explicit and precise
statement. However, these results are not readily applicable to
the exterior cohomology functor on  $\E^{\R_+}_w$. Indeed, on the
one hand,  it is not clear which family of objects is the one to
choose so that it compactly generates the exterior category. On
the other hand, as $\E^{\R_+}_w$ is not in general closed for
limits or colimits,  the required general form of the ``wedge''
and ``Mayer-Vietoris'' properties (see for instance G3 and G4 of
\cite[\S3]{jardine}) for the exterior cohomology functor on
cofibrant objects may not be attained.

We follow the original approach of Brown in \cite{Br}, with the
same notation and language, to prove Theorem \ref{intro} of the
Introduction in this section.

In what follows, an exterior based space $(X,\R_+)$ will often be
denoted simply by $X$.

We begin by proving all necessary  properties on the homotopy
category of (finite) based exterior CW-complexes.

\begin{lemma}\label{fundamentalcategory}
${}$
\begin{itemize}
\item[(i)] The category $\mathbf{Ho CW}^{\mathbb{R}_+}$
(resp. $\mathbf{Ho CW}^{\mathbb{R}_+}_f$) has arbitrary (resp.
finite) coproducts.

\item[(ii)] Any diagram in $\mathbf{Ho CW}^{\mathbb{R}_+}$
or $\mathbf{Ho CW}^{\mathbb{R}_+}_f$ of the form
${X_1} \stackrel{h_1}{\leftarrow}{A} \stackrel{h_2}{\rightarrow} {X_2}$  has a weak pushout.

\item[(iii)] Any given direct system in $\mathbf{Ho CW}^{\mathbb{R}_+}$,
$$\xymatrix{{X_1} \ar[r]^{f_1} & {X_2} \ar[r]^{f_2} & {X_3} \ar[r]^{f_3} \ar[r]
& ... \ar[r] & {X_n} \ar[r]^{f_n} & {X_{n+1}} \ar[r] & ... },$$
admits a weak colimit $(Y,\{g_n\}_{n\ge 1})$ for which the natural
maps
$$
[Y,Z]^{\mathbb{R}_+}\twoheadrightarrow
\mbox{lim}\hspace{3pt}[X_n,Z]^{\mathbb{R}_+}\qquad\text{and}\qquad \mbox{colim}\hspace{3pt}[Z,X_n]^{\mathbb{R}_+}\stackrel{\cong}{\rightarrow}
[Z,Y]^{\mathbb{R}_+}
$$
are, respectively, a surjection  for every $Z\in \mathbf{Ho CW}^{\mathbb{R}_+}$,
and a bijection for every $Z\in \mathbf{Ho CW}^{\mathbb{R}_+}_f$.

\end{itemize}
\end{lemma}

\begin{proof}

(i) Given any family $\{X_i,\}_{i\in I}$ of based exterior
CW-complexes, the exterior (homotopy) pushout
\begin{equation}\label{ecuacion2}
\xymatrix{ {\coprod _{i\in I}(\mathbb{R}_+)_i}
\ar@{^{(}->}[d] \ar[r] & {\mathbb{R}_+} \ar@{^{(}->}[d] \\
{\coprod _{i\in I}X_i} \ar[r] & {\vee _{i\in I}^{\mathbb{R}_+}X_i}
}
\end{equation}
 \noindent is easily checked to be their coproduct in
 $\mathbf{Ho CW}^{\mathbb{R}_+}$. Observe that, if each $(X_i,\mathbb{R}_+)$
 and $I$ are finite, then ${\vee _{i\in I}^{\mathbb{R}_+}X_i}\in \mathbf{CW}^{\mathbb{R}_+}_f$.

(ii) The homotopy pushout of $\xymatrix{{X_1} & {A} \ar[l]_{h_1}
\ar[r]^{h_2} & {X_2} }$ is obviously a weak pushout. For it to lie
within the category $\mathbf{Ho CW}^{\mathbb{R}_+}$ or $\mathbf{Ho
CW}^{\mathbb{R}_+}_f$, it is enough to choose exterior cellular
representatives (see \cite[Thms. 13,14]{G-G-M2}) of $h_1 $ and
$h_2$ and their factorization through exterior mapping cylinders.

(iii)  Again, choosing  exterior cellular representatives and
factorizations through exterior mapping cylinders we may assume,
without losing generality that each $f_n$ of the direct diagram is
a cellular cofibration. This way its colimit $(Y,\{g_n\}_{n\ge
1})$ is a based exterior CW-complex in which each $g_n\colon
X_n\rightarrowtail Y$ is a cofibration. Surjectivity of the first
map and injectivity of the second are obviously satisfied. For the
onto character of the second map it is enough to apply \cite[Prop.
4.1.21]{G} or \cite[Prop. 4.2]{Ext_1} which, for convenience, we
recall here: let
$$
\xymatrix{{X_1} \ar[r]^{f_1} & {X_2} \ar[r]^{f_2} & {X_3} \ar[r]^{f_3} \ar[r]
& ... \ar[r] & {X_n} \ar[r]^{f_n} & {X_{n+1}} \ar[r] & ... }
$$
be a sequence of  injective closed and e-closed exterior maps for
which the points  of $X_k\setminus X_1$ are also closed and
e-closed in $X_k$, for any $k$. As before, if we denote by
$(Y,\{g_n\}_{n\ge 1})$ the colimit of this sequence in the
exterior category, then each exterior map $h\colon Z\to Y$ from a
Hausdorff, locally compact, $\sigma $-compact space with the
cocompact externology, factors through $X_n$ for $n$ sufficiently
large, i.e. there is an exterior map $h_n\colon Z\to X_n$ such
that $g_nh_n=h$.
\end{proof}

Next, we see that, on $\mathbf{Ho CW}^{\mathbb{R}_+}$, the
exterior cohomology functor  satisfies the ``wedge'' and
``Mayer-Vietoris'' properties. For use in what follows, we remark
that for any cohomology theory $H^*$, the reduced functor
$\widetilde H^*$ on the based category satisfied these properties.

\begin{lemma}\label{lema1}
If $\{X_i\}_{i\in I}$ is any family of based exterior
CW-complexes, then there exists a canonical isomorphism
$$H^*_\exter(\vee _{i\in
I}^{\mathbb{R}_+}X_i)\cong \textstyle{\prod} _{i\in
I}H^*_\exter(X_i).$$
\end{lemma}

\begin{proof}
Applying the functor $(-)^{\infty }$ to \ref{ecuacion2} we obtain,
by Proposition \ref{propiedades}, a (homotopy) pushout in the
classical pointed category,
$$\xymatrix{ {\vee
_{i\in I}(\mathbb{R}_+)_i^{\infty }}
\ar@{^{(}->}[d] \ar[r] & {(\mathbb{R}_+)^{\infty }} \ar@{^{(}->}[d] \\
{\vee _{i\in I}X_i^{\infty }} \ar[r] & ({\vee _{i\in
I}^{\mathbb{R}_+}X_i)^{\infty }}. }$$
\noindent Indeed, note that
 the inclusion
$(\coprod _{i\in I}(\mathbb{R}_+)_i)^{\infty }\hookrightarrow
(\coprod _{i\in I}X_i)^{\infty }$ is homeomorphic to $\vee _{i\in
I}(\mathbb{R}_+)^{\infty }_i\hookrightarrow \vee _{i\in
I}X_i^{\infty }$. On the other hand, since the top row is
trivially a (classical) homotopy equivalence, so is the bottom
row. Applying this and Theorem \ref{igualdadcohomologia}  we
get,\hfill\break

\medskip\noindent
$
H^*_\exter(\vee _{i\in
I}^{\mathbb{R}_+}X_i)\cong\widetilde H^*\bigl((\vee
_{i\in I}^{\mathbb{R}_+}X_i)^{\infty }\bigr){\cong }
\widetilde H^*(\vee _{i\in I}X_i^{\infty
}){\cong } \textstyle{\prod}_{i\in
I}\widetilde H^*(X_i^{\infty } )\cong\textstyle{\prod}_{i\in
I}H^*_\exter(X_i).
$
\end{proof}

\begin{lemma}\label{lema2} If

$$\xymatrix{{A} \ar[d] \ar[r] & {X_2} \ar[d] \\
{X_1} \ar[r] & {X} }$$ is an exterior homotopy pushout of exterior
based CW-complexes, then the induced homomorphism
$$\xymatrix{
H^*_\exter(X) \ar[r]  &
H^*_\exter(X_1)\times
_{H^*_\exter(A)}H^*_\exter(X_2) }$$
is surjective
\end{lemma}

\begin{proof}
An argument by factorizations through exterior mapping cylinders
and exterior cellular approximations lets us assume, without
losing generality, that the above is a pushout in $\mathbf{Ho
CW}^{\mathbb{R}_+}$ where all maps are closed exterior
cofibrations, and therefore $A=X_1\cap X_2$ and $X=X_1\cup X_2.$
Applying the functor $(-)^{\infty }$, see Proposition
\ref{propiedades}, we obtain another pushout with analogous
properties in the based homotopy category
$$\xymatrix{{A^{\infty }} \ar@{^{(}->}[d]
\ar@{^{(}->}[r] & {X_2^{\infty }} \ar@{^{(}->}[d] \\
{X_1^{\infty }} \ar@{^{(}->}[r] & {X^{\infty }.} }$$ In other
words, $(X^{\infty };X_1^{\infty },X_2^{\infty })$ is a classical
excisive triad and the Mayer-Vietoris property for reduced
cohomology asserts, in particular, that the following
$$\xymatrix{
\widetilde H^*(X^{\infty }) \ar[r] & \widetilde H^*(X_1^{\infty
})\bigoplus \widetilde H^*(X_2^{\infty }) \ar[r]  & \widetilde
H^*(A^{\infty }) }$$ is exact. This, by Theorem
\ref{igualdadcohomologia}, is equivalent to saying that
$$\xymatrix{
H^*_\exter(X) \ar[r]  &
H^*_\exter(X_1)\times
_{H^*_\exter(A)}H^*_\exter(X_2) }$$
is surjective.
\end{proof}

We are now able to prove the main results of the paper.

\begin{proof}[Proof of Theorem \ref{intro}]
Lemmas \ref{lema1} and \ref{lema2} above show that the pair \break
$(\mathbf{Ho CW}^{\mathbb{R}_+},\mathbf{Ho CW}^{\mathbb{R}_+}_f)$
and the functors $H_\exter^n$, $n\ge 0$, are respectively, a  {\em
``homotopy category''} and {\em ``homotopy functors''} in the
sense of \cite[\S2]{Br}. In our setting, the first part of Theorem 2.8 of
\cite{Br}  translates to the fact that, for each
$n\geq 0$, there exists an exterior CW-complex $K_n\in \mathbf{Ho
CW}^{\mathbb{R}_+}$ and a universal element $u\in
H_\exter^n(K_n)$, such that the natural natural transformation
$$T_u:[-,K_n]^{\mathbb{R}_+}\longrightarrow H_\exter^n(-),\qquad T_u[f]=H_\exter^n(f)(u),$$
induces a bijection $ [X,K_n]^{\mathbb{R}_+}\cong H_\exter^*(X) $
for any exterior finite CW-complex $X\in \mathbf{Ho
CW}^{\mathbb{R}_+}_f$. To prove, also applying directly Theorem
2.8 of \cite{Br}, that $K_n$ is unique up to based exterior
homotopy, we see first that $K_n$ is e-path connected. On the one
hand, for the classical path-connectivity of $K_n$, consider the
following general situation: let $S^n_0$ denote the base ray
$\R_+$ with an $n$-sphere attached at $0$, i.e, the pushout of
$\R_+\leftarrow \{0\}\rightarrow S^n$. If $\theta$ denotes the
obvious base ray of $S^n_0$ and $(X,\alpha)$ is any based exterior
space, then there is an isomorphism,
$$
[(S^n_0,\theta ),(X,\alpha )]^{\mathbb{R}_+}=\pi_n\bigl(X,\alpha(0)\bigr).
$$
Therefore
$$\pi _0(K_n)\cong [S^0_0,K_n]^{\mathbb{R}_+}\cong H^n_{\exter }(S^0_0)\cong
\widetilde{H}^n((S^0_0)^+)\cong 0.$$
On the other hand,
 $$
\pi _0^{\mathcal{B}}(K_n)\cong [\SB^0,K_n]^{\mathbb{R}_+}\cong
H_\exter^n(\SB^0)\cong
{\widetilde H^n\bigl(({\SB^0})^+\bigr)}\cong 0.
$$

We finish by remarking that the subcategory of $\mathbf{Ho
CW}^{\mathbb{R}_+}$ consisting of\break e-path connected, exterior
CW-complexes  is compactly generated by\break $\mathbf{Ho
CW}^{\mathbb{R}_+}_f$. That is, a map $f\colon X\to Y$ between
e-path connected exterior CW-complexes is an exterior based
homotopy equivalence if and only if $f_*\colon [Z,X]^{\R_+}\to
[Z,Y]^{\R_+}$ is a bijection for every  $Z\in
\mathbf{CW}^{\mathbb{R}_+}_f$. In fact, by the {\em exterior
Whitehead theorem} \cite[Thm. 12]{G-G-M2}, as well as its
classical version, e-path connected, based, exterior CW complexes
are compactly generated simply by the finite based exterior
CW-complexes $S_\eta^n$ and $S^n_0$, $n\ge 1$. For it, take into
account that
$$[(S^n_\eta,\xi), (X,\alpha)]^{\R_+} \cong [(\SN^n,\eta),(X,\alpha)]^\N=\piB_n(X,\alpha)$$
 and  the above isomorphism,
$$
[(S^n_0,\theta ),(X,\alpha )]^{\mathbb{R}_+}=\pi_n\bigl(X,\alpha(0)\bigr)
$$
for any based exterior space $(X,\alpha)$. This completes the
proof.
\end{proof}

In the special case of considering classical cohomology with compact supports we obtain:

\begin{proof}[Proof of Theorem \ref{intro2}]
Simply note that any countable, locally finite, finite dimensional
relative CW-complex $(X,\R_+)$ is an exterior finite based
CW-complex endowed with the cocompact externology, see
\cite{garcalpari}, \cite[\S2.1]{G-G-M2} or \cite[\S5.B]{Ext_1}).
\end{proof}

We finish by characterizing, as in the classical case, the
classifying space  $K_n$. For it recall that, given an integer
$m\ge1$ and a group $G$ (abelian if $m\ge 2$), we will denote by
$K_B(G,m)$ the {\em Eilenberg-MacLane exterior space $K_B(G,m)$
for Brown-Grossman homotopy groups} \cite{exherri}, whose homotopy
type in $\mathbf{ CW}^{\mathbb{R}_+}$ is unique  due to the
exterior Whitehead theorem \cite[Thm. 12]{G-G-M2}. In what
follows, $H^*$ will denote either singular, \v{C}ech,
Alexander-Spanier, or any other isomorphic cohomology theory on
HLC spaces. In fact, we only need $H^*_\exter$ to be isomorphic to
any of the above on $S^m_\eta$, for any $m\ge 1$. In this context
we are also allowed to consider a coefficient ring $R$ which is
fixed henceforth. $R^\infty $ will denote the product of countably
infinitely many copies of $R.$

\begin{theorem}\label{descripcion}  For any $n\ge 1$,
the classifying space $K_n$ is exterior homotopy equivalent to $K_B(R^\infty,n)$.
\end{theorem}

\begin{proof}
In view of Theorem \ref{intro}, we have a bijection,
$$
 \piB_m(K_n,\alpha)=[(S^m_\eta,\xi), K_n]^{\R_+}\cong H_\exter^n(S^m_\eta)\cong H_c^n(S^m_\eta)\cong\begin{cases}R^\infty,&{n=m;}\\0, &n\not=m.\end{cases}
$$
To prove that this is in fact an isomorphism it is enough to check
that exterior cohomology respects the addition in
$\piB_m(K_n,\alpha)$, that is,
$H_\exter^n(f+g)=H_\exter^n(f)+H_\exter^n(g)$ for any $f,g\in
\piB_n(K_n,\alpha)$.  More generally, for every based exterior
CW-complex $X$ and any par of exterior maps $f,g\colon S^m_\eta\to
X$ representing elements of $\piB_m( X)$ we show that
$$
H^n_\exter(f+g)=H^n_\exter(f)+H^n_\exter(g)\colon H^n_\exter(X)\longrightarrow H^n_\exter(S^m_\eta).
$$
Here, abusing of notation, we do not distinguish maps from the
homotopy classes which they represent. Due to the representability
of the exterior cohomology, the above is equivalent to seeing that
$$
(f+g)^*=f^*+g^*\colon [X,K_n]^{\R_+}\longrightarrow [S^m_\eta,K_n]^{\R_+},\qquad
$$
where $(-)^*$ denotes composition on the right.

For it, consider the isomorphism
$$
w_X\colon \piB_m(X)\cong[S^m_\eta,X]^{\R_+}\cong[S^m,X^{\N}]=\pi_m(X^\N)
$$
 and observe that its naturality provides, for every $h\in[X,K_n]^{\R_+}$, a commutative square
 $$
\xymatrix{
  \piB_m(X) \ar[d]_{w_X}^\cong \ar[r]^{h_*} & \piB_m(X) \ar[d]^{w_{K_n}}_\cong \\
  \pi_m(X^\N)\ar[r]_{(h^\N)_*} & \pi_m(K_n^\N). }
$$
Here $(-)_*$ denotes composition on the left. Therefore
$$
w_{K_n}\bigl(h(f+g)\bigr)=w_{K_n}h_*(f+g)=(h^\N)_*w_X(f+g)=w_{K_n}(hf+hg).
$$
\noindent and $h(f+g)=hf+hg,$ as required.
\end{proof}

\bigskip
\bigskip\bigskip
\noindent{\sc J. M. Garc\'{\i}a-Calcines, P. R. Garc\'{\i}a-D\'{\i}az\hfill\break Departamento de Matem\'atica
Fundamental, Universidad de La Laguna, 38271 La Laguna, Spain.}\hfill\break
\texttt{jmgarcal@ull.es, prgdiaz@ull.es}}
\hfill\break

\medskip
\noindent {\sc A. Murillo \hfill\break Departamento de \'Algebra, Geometr\'{\i}a y Topolog\'{\i}a, Universidad de M\'alaga, Ap.\ 59, 29080 M\'alaga, Spain}.\hfill\break 
\texttt{aniceto@agt.cie.uma.es}

\end{document}